\documentclass[12pt]{amsart}

\usepackage{amsfonts,amsmath,amsthm}
\usepackage{latexsym}

\newtheorem{theorem}{Theorem}[section]
\newtheorem{corollary}[theorem]{Corollary}
\newtheorem{lemma}[theorem]{Lemma}

\newtheorem{conjecture}[theorem]{Conjecture}

\newtheorem{definition}{Definition}[section]

\theoremstyle{definition}

\newcommand\ran{\mathrm{Ran}\:}
\newcommand{\TT}{\ensuremath{\mathbb{T}}}
\newcommand{\Q}{\ensuremath{\mathbb{Q}}}
\newcommand{\ZZ}{\ensuremath{\mathbb{Z}}}
\newcommand{\DD}{\ensuremath{\mathbb{D}}}
\newcommand{\R}{\ensuremath{\mathbb{R}}}
\newcommand{\K}{\ensuremath{\mathbb{K}}}
\newcommand{\Co}{\ensuremath{\mathbb{C}}}

\def \e {\varepsilon}

\def \< {\langle}
\def \> {\rangle}

\def \span {{\rm span}}
\def \dim  {{\rm dim}}

\def \vu {{\mathbf u}}
\def \vx {{\mathbf x}}
\def \vq {{\mathbf q}}

\def \vy {{\mathbf y}}
\def \vv {{\mathbf v}}
\def \vw {{\mathbf w}}

\def \ma {{\mathbf A}}

\begin{document}

\title[A note on the pyjama problem]{A note on the pyjama problem}

\author[R. D. Malikiosis]{R. D. Malikiosis}
\address{R. D. M.: School of Physical and Mathematical Sciences,
Nanyang Technological University, 21 Nanyang Link, 637371
Singapore}
\email{rmalikiosis@ntu.edu.sg}

\author[M. Matolcsi]{M. Matolcsi}
\address{M. M.: Alfr\'ed R\'enyi Institute of Mathematics,
Hungarian Academy of Sciences POB 127 H-1364 Budapest, Hungary
Tel: (+361) 483-8307, Fax: (+361) 483-8333}
\email{matomate@renyi.hu}

\author[I. Z. Ruzsa]{I. Z. Ruzsa}
\address{I. R.: Alfr\'ed R\'enyi Institute of Mathematics,
Hungarian Academy of Sciences POB 127 H-1364 Budapest, Hungary
Tel: (+361) 483-8328, Fax: (+361) 483-8333}
\email{ruzsa@renyi.hu}

\thanks{M.M. and I.Z.R. were supported by the ERC-AdG 228005,
and OTKA Grants No. K81658, and M.M. also by the Bolyai
Scholarship. R.D.M. was partially supported by Singapore MOE grant MOE2011-T2-1-090.}

%%% \date{\today}

\begin{abstract}
This note concerns the so-called pyjama problem, whether it is
possible to cover the plane by finitely many rotations of vertical
strips of half-width $\varepsilon$. We first prove that there
exist no periodic coverings for $\e<\frac{1}{3}$. Then we describe
an explicit (non-periodic)  construction for $\e
=\frac{1}{3}-\frac{1}{48}$. Finally, we use a compactness argument
combined with some ideas from additive combinatorics to show that
a finite covering exists for $\e =\frac{1}{5}$. The question
whether $\e$ can be arbitrarily small remains open.
\end{abstract}

\maketitle

\bigskip

\section{introduction}

This note concerns a question that has been advertised as the
"pyjama-problem" in the additive combinatorics community. The
problem was originally  raised in \cite{ikm}, and we recall it
here for convenience. Let $\|x\|$ denote the distance of any real
number $x$ to the closest integer, and define the following set of
equidistant vertical strips of width $2\e$ on $\R^2$:
\[E_{\e}:=\{(x,y)\in \R^2 : \|x\|\le \e\}.\]
Denote by $R_{\theta}$ the counterclockwise rotation of the plane
by angle $\theta$ (around the origin). The question is whether we
can cover the plane by the union of finitely many rotates of
$E_\e$, i.e. whether there exist angles $\theta_0, \dots \theta_n$
such that $\R^2=\bigcup_{j=0}^nR_{\theta_j}E_{\e}.$ We will assume
throughout this note (without loss of generality) that the angles
$\theta_0, \dots \theta_n$ are pairwise distinct.

\medskip

We make a few remarks on the origin of the problem. In
\cite{fkw} Furstenberg, Katznelson and Weiss proved that for any
set $A$ of positive upper density in $\R^2$ there exists a
threshold $t_0\in \R$ such that for any $t\ge t_0$  there exist
points in $A$ with distance $t$. Another proof was given by
Falconer and Marstrand in \cite{fm}. Subsequently, Bourgain
\cite{bourgain} used a Fourier analytic argument to generalize the
result in higher dimensions: a set $A$ of positive upper density
in $\R^k$ contains all large enough copies of any $k-1$
dimensional simplex. Kolountzakis \cite{kolountzakis} used a
similar argument to prove that even if the unit ball is
non-Euclidean (but smooth), all large enough distances appear
between points of $A$. In this circle of problems, a natural
question of Sz. R\'ev\'esz (private communication) was whether it
is true in $\R^2$ that finitely many {\it rotations} of the difference
set $A-A$ can cover the plane, with the exception of a bounded
set. This would have given a generalization of the above mentioned
result of \cite{fm, fkw}. However, the question was answered in
the negative in \cite{ikm}, where $A$ was taken to be a set of
small disks around points of the integer lattice.  The
"pyjama-problem" was formulated and left open in \cite{ikm}. For
further related problems concerning extremal properties of sets of
positive upper density we refer to \cite{reveszdoktori}.

\medskip

Let us introduce some notations and definitions.

\begin{definition}\label{def1}
We will say that $\e$ has the {\it finite rotation property} if
there exist angles $\theta_0, \dots \theta_n$ such that
$\R^2=\bigcup_{j=0}^nR_{\theta_j}E_{\e}.$ Let $\e_0$ denote the
infimum of the values of $\e$ having the finite rotation property.
\end{definition}

We believe that $\e$ can be arbitrarily small.

\begin{conjecture}
With notation introduced above, we have $\e_0=0$.
\end{conjecture}

Let $\vu_j=(\cos \theta_j, \sin \theta_j)\in \R^2$ be the unit
vector corresponding to the angle $\theta_j$. It is easy to see
that a vector $\vx\in \R^2$ is covered by $R_{\theta_j}E_{\e}$ if
and only if $\| \langle \vu_j, \vx \rangle \| \le \e$. The pyjama
problem can therefore be formulated in an equivalent way as
follows:

\medskip

\noindent For a given $\e >0$ we want to find unit vectors $\vu_1,
\dots \vu_n\in \R^2$ such that for all $\vx \in \R^2$ there exists
a $\vu_j$ such that $\| \langle \vu_j, \vx \rangle \| \le \e$.

\medskip

The case $\e=\frac{1}{3}$ is "trivial". Indeed, the rotations by
angles $0, \frac{2\pi}{3}, \frac{4\pi}{3}$ will suffice, as the
reader can easily verify. As we shall see, it is not at all
trivial to go below $\e=\frac{1}{3}$. The above covering by
rotations $0, \frac{2\pi}{3}, \frac{4\pi}{3}$ is {\it periodic}.
One natural approach is to consider other periodic arrangements of
the strips (e.g. angles corresponding to Pythagorean triples). We
will make the concept of periodicity rigorous in Section
\ref{sec:per}, and prove that it can never work for any
$\e<\frac{1}{3}$.

\medskip

Another natural approach is to consider angles corresponding to
$N$th roots of unity for some $N$. It can be proven, however,
that we will not get a covering of $\R^2$ in
this manner for any $N$ and any $\e<\frac{1}{3}$. We will not include the
proof of this negative result to keep this note brief.

\medskip

A random set of angles will not lead to a covering for any
$\e<\frac{1}{2}$, almost surely. The reason is that the numbers
$\cos \theta_0, \dots, \cos\theta_n$ will be almost surely
independent over $\Q$, and therefore the set $\R (\cos
\theta_0, \dots, \cos\theta_n)$, mod 1, will be dense in the
torus $\TT^{n+1}$, and thus we can find a vector $\vx=(x,0)\in
\R^2$ such that $\| \langle \vu_j, \vx \rangle \| \approx
\frac{1}{2}$ for all $j$. A similar argument shows that we
will not get a covering for any $\e<\frac{1}{2}$ if the vectors
$\vu_0, \dots, \vu_n$ are independent over $\Q$ (in that case one
needs to consider $\vx=(x\cos \varphi, x\sin \varphi)$ for some
appropriate angle $\varphi$ and $x\in \R$).

\medskip

In Section \ref{sec:constr} we will show a specific finite set of
angles for $\e=\frac{1}{3}-\frac{1}{48}$. Finally, in Section
\ref{sec:compact} we will use a compactness argument combined with
some ideas from additive combinatorics to show that a finite
covering exists for $\e =\frac{1}{5}$. However, this result is
{\it non-constructive}, i.e. we will not be able to exhibit the
appropriate angles.

\section{Periodic covering is not possible for $\e<\frac{1}{3}$}\label{sec:per}

Periodicity is a natural idea to ensure that the whole plane gets
covered by the rotated strips. The reason is that in this case
only the fundamental region (which is a finite parallelogram,
spanned by the period-vectors) needs to be checked. However, we
will now show that $\e$ cannot be smaller than $\frac{1}{3}$ in
such a case.

\begin{definition}\label{defper}
Given $\theta_0, \dots, \theta_n$, and the corresponding unit
vectors $\vu_0, \dots, \vu_n$, a vector $\vv\in \R^2$ is a
period-vector of the set $\bigcup_{j=0}^nR_{\theta_j}E_{\e}$, if
$\langle \vu_j, \vv \rangle \in \ZZ$ for each $j$. The set
$\bigcup_{j=0}^nR_{\theta_j}E_{\e}$ is called (fully) periodic if
it has two $\R$-linearly independent period vectors $\vv_0, \vv_1$.
\end{definition}

Periodicity is directly related to the dimension of the space
spanned by the vectors $\vu_0, \dots, \vu_n$ over $\Q$.

\begin{lemma}\label{lemmaper}
For any $n\ge 1$ the set $\bigcup_{j=0}^nR_{\theta_j}E_{\e}$ is
periodic if and only if $d:=\dim \left(\span \{\vu_0, \dots ,
\vu_n\}_{\Q}\right)=2.$ \end{lemma}
\begin{proof}
Assume $d=2$. As $\vu_0$ and $\vu_1$ are distinct, we can find two
$\R$-linearly independent vectors $\vv_0, \vv_1$ such that $\langle
\vu_j, \vv_k \rangle \in \ZZ$ for $0\le j,k\le 1$. Let
$\vu_2=q_0\vu_0+q_1\vu_1$ where $q_0, q_1\in \Q$, and let $M$
denote the least common multiple of the denominators of $q_0,
q_1$. Then it is straightforward to check that the vectors
$\vw_0=M \vv_0,  \ \vw_1=M \vv_1$ are two period-vectors with
respect to $\vu_0, \vu_1, \vu_2$. We can then proceed by induction
to produce two period vectors with respect to $\vu_0, \dots
\vu_n$.

\medskip

Assume now that $\bigcup_{j=0}^nR_{\theta_j}E_{\e}$ is periodic,
i.e. there exist two linearly independent period vectors $\vv_0,
\vv_1$ such that $\langle \vu_j, \vv_k \rangle := m_{j,k} \in \ZZ$
for each $j=0, \dots, n$ and $k=0,1$. We need to prove that each
$\vu_j$ is a $\Q$-linear combination of $\vu_0$ and  $\vu_1$. Let
us fix $j$. The $2\times 2$ matrix $\ma$ given by
$a_{k,r}=m_{r,k}$ (for $0\le k,r\le 1$) contains integer entries
and is non-singular. Therefore, there exists a vector
$\vq_j=\binom {q_{0,j}}{q_{1,j}}$ with rational coordinates such
that $\ma . \vq_j =\binom{-m_{j,0}}{-m_{j,1}}$. But then the
vector $\vu_j+q_{1,j}\vu_1+q_{0,j}\vu_0$ is orthogonal to both
$\vv_0$ and $\vv_1$, and therefore must be zero.
\end{proof}

\begin{lemma}\label{qextension}
If $n\ge 2$, $\theta_0=0$ and $\bigcup_{j=0}^nR_{\theta_j}E_{\e}$
is periodic, then all $e^{i\theta_j}$ (for $j=0,\dots, n$) belong
to the same quadratic imaginary field.
\end{lemma}
\begin{proof}
As we saw in the previous lemma, there are non-zero rational numbers $q_0$,
$q_1$, such that $q_0+q_1e^{i\theta_1}=e^{i\theta_2}.$ Hence
$\lvert q_0+q_1e^{i\theta_1}\rvert^2=1$, which yields that
$\cos\theta_1=\frac{1-q_0^2-q_1^2}{2q_0q_1}$ is rational (and
similarly, all $\cos\theta_j$ are rational).  Furthermore,
$e^{i\theta_1}$ satisfies the quadratic equation
$x^2-2\cos\theta_1x+1=0$. As all $e^{i\theta_j}$ are $\Q$-linear
combinations of $e^{i\theta_0}=1$ and $e^{i\theta_1}$, they all
belong to the same imaginary quadratic field $\Q(e^{i\theta_1})$.
\end{proof}

We are now in position to prove the main (negative) result of this
section.

\begin{theorem}
If $\e<\frac{1}{3}$ and $\bigcup_{j=0}^nR_{\theta_j}E_{\e}$ is
periodic, then it does not cover the whole plane.
\end{theorem}
\begin{proof}
If $n=0$ or $1$ then it is trivial that covering is not possible
for  any $\e<\frac{1}{2}$.

\medskip

Assume $n\ge 2$, and $\bigcup_{j=0}^nR_{\theta_j}E_{\e}$ is
periodic. Without loss of generality we can assume that
$\theta_0=0$. Then all the numbers $e^{i\theta_j}$ (for
$j=0,\dots, n$) belong to the same quadratic imaginary field, and
all the numbers  $\cos\theta_j$ are rational by Lemma
\ref{qextension}. Let the field be denoted by $\K=\Q (\sqrt{-D})$,
where $D$ is a positive square-free integer. Then $(\cos\theta_j,
\sin\theta_j)=(\frac{m_j}{n_j}, \frac{k_j\sqrt{D}}{n_j})$ for some
integers $m_j, k_j, n_j$ such that $m_j^2+Dk_j^2=n_j^2$. We may assume
that $\textrm{gcd}(m_j, k_j, n_j)=1$; then the equation $m_j^2+Dk_j^2=n_j^2$ implies
that $m_j, k_j, n_j$ are pairwise co-prime. Let $M=\prod_{j=0}^{n}n_j$.

\medskip

If $D=1$ then the angles $\theta_j$ correspond to Pythagorean
triples, $m_j^2+k_j^2=n_j^2$. Note that all the $n_j$ must be odd.
Consider the point $\vx=(\frac{M}{2},\frac{M}{2})\in \R^2$. Then
$ \langle \vu_j, \vx \rangle =
\frac{M}{n_j}\left(\frac{m_j+k_j}{2}\right)$, where
$\frac{M}{n_j}$ is an odd integer and $m_j+k_j$ is odd. Therefore,
$\| \langle \vu_j, \vx \rangle \|=\frac{1}{2}$ for each $j$, and
hence $\vx$ is not covered if $\e<\frac{1}{2}$.

\medskip

If $D\ne 1$ then let $p$ denote the smallest prime dividing $D$.
If $p=2$ then all $m_j, n_j$ must be odd. Consider the point
$\vx=(\frac{M}{2},M\sqrt{2})\in \R^2$. Then $ \langle \vu_j, \vx
\rangle =\frac{M}{n_j}\frac{m_j}{2}+\frac{2Mk_j}{n_j}$.
Therefore $\| \langle \vu_j, \vx \rangle \|=\frac{1}{2}$ for each
$j$, and hence $\vx$ is not covered if $\e<\frac{1}{2}$.

\medskip

Finally, assume $p>2$.  Then $m_j^2+Dk_j^2=n_j^2$ implies that
$m_j \equiv \pm n_j$ (mod $p$), and all $m_j, n_j$ are relatively
prime to $p$. Let $t$ be an integer such that $tM\equiv 1$ mod
$p$. Consider the point $\vx=(\frac{tM(p-1)}{2p},M\sqrt{D})\in
\R^2$. Then $ \langle \vu_j, \vx \rangle
=\frac{tMm_j}{n_j}\frac{p-1}{2p}+ \frac{k_jMD}{n_j}$. Here
$\frac{tMm_j}{n_j}$ is an integer which is $\pm 1$ (mod $p$), and
$\frac{k_jMD}{n_j}$ is also an integer. Therefore, $\| \langle
\vu_j, \vx \rangle \|=\frac{p-1}{2p}\ge \frac{1}{3}$ for each $j$,
and hence $\vx$ is not covered if $\e<\frac{1}{3}$.
\end{proof}

\section{Covering with $\e =\frac{1}{3} -\frac{1}{48}$}\label{sec:constr}

Having all the negative results so far, one might be tempted to
conjecture that the trivial covering cannot be improved and
$\e_0=\frac{1}{3}$. However, we will now show that this is not the case.

\medskip

\begin{theorem}
Let $\e =\frac{1}{3} -\frac{1}{48}$. Define $\theta_1 =0$, and
$\theta_2=\frac{2\pi}{3}, \theta_3=\frac{4\pi}{3}$. Let $\theta_4$
be such that $(\cos\theta_4, \sin\theta_4)=(\frac{1}{3},
\frac{\sqrt{8}}{3})$, and $\theta_5=\theta_4+\frac{2\pi}{3},
\theta_6=\theta_4+\frac{4\pi}{3}$. Let $\theta_7$ be such that
$(\cos\theta_7, \sin\theta_7)=(\frac{1}{3}, -\frac{\sqrt{8}}{3})$,
and $\theta_8=\theta_7+\frac{2\pi}{3},
\theta_9=\theta_7+\frac{4\pi}{3}$. Then
$\bigcup_{j=1}^9R_{\theta_j}E_{\e}=\R^2$.
\end{theorem}

\medskip

\begin{proof}

Let $\vx\in \R^2$ be arbitrary, and let $\langle \vu_j, \vx
\rangle=s_j$ for $1\le j\le  9$. Observe the following relations:

\begin{equation*}
s_1+s_2+s_3=0, \ \ s_4+s_5+s_6=0, \ \ s_7+s_8+s_9=0, \ \
3(s_4+s_7)-2s_1=0.
\end{equation*}

\medskip

Let $\ma$ denote the $4\times 9$ matrix corresponding to this set
of linear equations. Assume, by contradiction, that $\|s_j\|
> \e$ for all $j$. Then the fractional parts  $w_j=\{s_j\}$ must lie in
the interval $I_{\e}=(\frac{1}{3} -\frac{1}{48}, \frac{2}{3}
+\frac{1}{48})$. Therefore the vector $\vw=(w_1, \dots, w_9)$ is
contained in the cube $I_{\e}^9$, and the image of $\vw$ under the
linear transformation $\ma$ must be an integer lattice point in
$\R^4$. We will show that this is not possible.

\medskip

We claim that all the $w_j$ must fall into $I_1\cup I_2$, where
$I_1=(\frac{1}{3} -\frac{1}{48}, \frac{1}{3}+\frac{1}{24})$ and
$I_2=(\frac{2}{3}-\frac{1}{24}, \frac{2}{3}+\frac{1}{48})$.
Indeed, $w_1\equiv -w_2-w_3$ (mod 1), and $-w_2-w_3\in -I_\e
-I_\e=(-\frac{4}{3}-\frac{1}{24}, -\frac{2}{3}+\frac{1}{24}
)\equiv [0, \frac{1}{3}+\frac{1}{24})\cup
(\frac{2}{3}-\frac{1}{24}, 1)$ (mod 1), and hence $w_1$ must fall
into the intersection of this set with $I_\e$ which is exactly
$I_1\cup I_2$. The same reasoning works for all $w_j$.

\medskip

Finally, if all the $w_j$ fall into $I_1\cup I_2$, then the
equation $3(w_4+w_7)-2w_1\equiv 0$ (mod 1) cannot be satisfied.
The reason for this is that $\|3(w_4+w_7)\|<\frac{1}{4}$ (because
$w_4, w_7\in I_1\cup I_2$), while $\|-2w_1\|>\frac{1}{4}$ (because
$w_1\in I_1\cup I_2$).
\end{proof}

This construction can be improved to decrease the value of $\e$.
However,  we do not see any argument to show that $\e$ can be
arbitrarily close to zero.

\section{A compactness argument for $\e =\frac{1}{5}$}\label{sec:compact}

We now turn to a non-constructive compactness argument which
allows us to decrease the value of $\e$.

\medskip

\begin{lemma}\label{compact}
Let $\TT$ denote the group $[-\frac{1}{2},\frac{1}{2})$ with the
addition operation mod 1. If $\e$ does not have the finite
rotation property then there exists a non-continuous additive
homomorphism $\gamma : \R^2 \to \TT$ such that $|\gamma(\vu)|\ge
\e$ for all unit vectors $\vu$. Conversely, if $\e$ has the finite
rotation property then there exists no additive homomorphism
$\gamma : \R^2 \to \TT$ such that $|\gamma(\vu)|> \e$ for all unit
vectors $\vu$.
\end{lemma}
\begin{proof}
Let $\Gamma$ denote the dual group of $\R^2$ ($\R^2$ is meant here
as an additive group with the Euclidean topology). Then $\Gamma$
can be identified with $\R^2$ in the usual way, $\vx
\leftrightarrow \gamma_\vx$ where $\gamma_\vx : \R^2 \to \TT$ is
the character $\gamma_\vx (\vu):=\langle \vx,\vu \rangle$ (mod 1).
Now, consider $\R^2$ as an additive group with the discrete
topology. Then its dual group, denoted by $\Gamma'$, is compact
and consists of all possible additive homomorphisms from $\R^2 \to
\TT$.

\medskip

Let $\Co_1$ denote the unit circle in the plane $\R^2$. The
assumption that $\e$ does not have the finite rotation property
means that for any $\vu_1, \dots, \vu_N\in \Co_1$ there exists an
$\vx\in \R^2$ such that $\|\langle \vu_j, \vx\rangle \|\ge \e$ for
each $\vu_j$. In other words, there exists a $\gamma_\vx \in
\Gamma \subset \Gamma'$ such that $| \gamma_\vx (\vu_j) | \ge \e$
for each $\vu_j$. Now, due to the compactness of $\Gamma'$ we
claim that there must exist $\gamma \in \Gamma'$ such that $|
\gamma | \ge \e$ on the whole of $\Co_1$. Indeed, this is the
so-called finite intersection property of compact sets: if
$F_{\vu}$ denotes the set of characters $\gamma\in \Gamma'$ such
that $| \gamma(\vu ) |\ge \e$ then our condition says that any
finite intersection of such sets $F_{\vu}$ is non-empty. Note that
$F_{\vu}$ are closed sets, and therefore the intersection of all
of the sets $F_{\vu}$ is non-empty by compactness.

\medskip

We now prove the converse statement. If $\e$ has the finite
rotation property then there exist unit vectors $\vu_1, \dots,
\vu_n$ such that for every $\vx\in \R^2$ we have $\|\langle \vu_j,
\vx \rangle \| \le \e$ for some $j$. Let $M\subset \ZZ^n$ describe
the rational linear relations among the vectors: $M=\{(m_1, \dots,
m_n)\in \ZZ^n : \sum_j m_j\vu_j=0\}$. Let $g: \R^2 \to \TT^n$ be
the function defined by $g(\vx)=(\langle \vu_1, \vx \rangle,
\dots, \langle \vu_n, \vx \rangle)$, and let
$S=\overline{\ran(g)}\subset \TT^n$ denote the closure of the
range of $g$. Then $S$ is a closed subgroup, and $S\cap (\e,
1-\e)^n=\emptyset$. The subgroup $S$ is characterized by the
linear relations in $M$, namely $S=\{(x_1, \dots , x_n)\in \TT^n :
\sum_j m_jx_j=0 \ \textrm{for all} \ (m_1, \dots, m_n)\in M\}$.

\medskip

Consider any additive homomorphism $\gamma : \R^2 \to \TT$. For
every $\vx\in \R^2$ we have $(\gamma(\vu_1), \dots ,
\gamma(\vu_n))\in S$ because $\sum_j
m_j\gamma(\vu_j)=\gamma(\sum_j m_j\vu_j)=0$. As $S\cap (\e,
1-\e)^n=\emptyset$ we conclude that $|\gamma(\vu_j)|\le \e$ for
some $j$.
\end{proof}

In the following auxiliary result it will be convenient to
identify $\R^2$ with $\Co$.

\begin{lemma}\label{prop1}
Let $\vu_1, \dots \vu_n\in \Co$ be unit vectors, and let $M\subset
\ZZ^n$ describe the rational linear relations among them:
$M=\{(m_1, \dots, m_n)\in \ZZ^n : \sum_j m_j\vu_j=0\}$. Let
$S\subset \TT^n$ be the subgroup $S=\{(x_1, \dots , x_n)\in \TT^n
: \sum_j m_jx_j=0 \ \textrm{for all} \ (m_1, \dots, m_n)\in M\}$.
Let $\gamma : \Co \to \TT$ be a non-continuous additive
homomorphism, and let $g: \Co \to \TT^n$ be the function defined
by $g(z)=(\gamma(\vu_1 z), \dots  , \gamma(\vu_n z))$. Let $U$ be
any neighbourhood of zero. Then $g(U)$ is dense in $S$.
\end{lemma}
\begin{proof}
We use the standard notation $\overline{A}$ for the closure of a
set $A$. We will show that as $U$ runs through the neighbourhoods
of zero, we have $B:=\cap_{U} \ \overline{g(U)}=S$ (this is
clearly equivalent to the statement of the proposition). First
notice that $B$ is a non-empty ($0\in B$) and compact set. It is
also clear that $B\subset S$, as $g(z)\in S$ for every $z\in \Co$.
We claim that $B$ is a subgroup of $S$. To see this, let
$b_1,b_2\in B$, and let $U$ be any neighbourhood of zero. We want
to show that $b_1-b_2\in \overline{g(U)}$. Let $U'$ be a smaller
neighbourhood, such that $U'-U'\subset U$. Then $b_1-b_2\in
\overline{g(U')}-\overline{g(U')}=\overline{g(U')-g(U')}=\overline{g(U'-U')}\subset
\overline{g(U)}$. Therefore $B$ is a subgroup.

\medskip

If $B=S$ then we are done. If $B$ is a compact non-empty proper
subgroup of $S$, then there exists a continuous character $\chi:
\TT^n \to \TT$ such that $\chi$ is not identically zero on $S$,
but $\chi|_B\equiv 0$. Such a character $\chi$ can be identified
with an $n$-tuple of integers, $\chi=(a_1,\dots a_n)\in \ZZ^n$,
$\chi \notin M$, so that $\chi(t_1, \dots t_n)=a_1t_1+\dots
a_nt_n$ for any $(t_1, \dots t_n)\in \TT^n$. Let $\beta =\sum_j
a_j\vu_j$, which is non-zero because $(a_1,\dots a_n)\notin M$.
Consider the additive homomorphism $h: \Co \to \TT$ defined by
$h(z)=\gamma(\beta z)=\chi(g(z))$.  We claim that $h(z)$ is
continuous. Assume it is not. Then there exists a sequence $z_m\to
0$ such that $h(z_m)$ does not converge to 0. By passing to a
subsequence, we may assume that $h(z_m) \to w\ne 0$. Again, by
passing to a subsequence we may assume that $g(z_m)$ converges to
some $\vy\in \TT^n$. But then $\vy\in B$ and $\chi (\vy)=\lim \chi
(g(z_m))=\lim h(z_m)=w\ne0$, a contradiction. Therefore, $h(z)$ is
continuous, and so is $\gamma(z)=h(\frac{1}{\beta}z)$, which
contradicts our assumption on $\gamma$.
\end{proof}

With the help of Lemma \ref{compact} and \ref{prop1} we can prove
the main result of this section: even if
$\bigcup_{j=0}^nR_{\theta_j}E_{\e}$ does not cover $\R^2$ but the
non-covered region is small enough in some sense, we can still
conclude that $\e_0\le \e$.

\begin{theorem}\label{thm:main}
Assume $\e>0$ and unit vectors $\vu_1, \dots, \vu_n\in \Co$ are
given. Let $M\subset \ZZ^n$ describe the rational linear relations
among the vectors: $M=\{(m_1, \dots, m_n)\in \ZZ^n : \sum_j
m_j\vu_j=0\}$. Let $S\subset \TT^n$ be the subgroup $S=\{(x_1,
\dots , x_n)\in \TT^n : \sum_j m_jx_j=0 \ \textrm{for all} \ (m_1,
\dots, m_n)\in M\}$, and let $X_\e =S\cap (\e, 1-\e)^n$. If the
difference set $X_\e-X_\e$ is not dense in $S$ then $\e_0\le \e$,
with the notation of Definition \ref{def1}.
\end{theorem}

\begin{proof}
Let $\e'>\e$ arbitrary, and assume by contradiction that $\e'$
does not have the finite rotation property. Then there exists a
non-continuous additive homomorphism $\gamma:\R^2\to \TT^n$ such
that $|\gamma(\vu)|\ge \e'$ for all unit vectors $\vu$, by Lemma
\ref{compact}. Let $g: \Co \to \TT^n$ be the function defined by
$g(z)=(\gamma(\vu_1 z), \dots  , \gamma(\vu_n z))$. Then $g(z)\in
S\cap [\e', 1-\e']^n\subset X_\e$ for all unit vectors $z$.
However, if $\DD$ denotes the closed unit disk then
$g(2\DD)=g(\Co_1)-g(\Co_1)\subset X_\e-X_\e$ should be dense in
$\TT^n$ by Lemma \ref{prop1}, a contradiction.
\end{proof}

Improved upper bounds on $\e_0$ follow immediately.

\begin{corollary}\label{onefourth}
With the notation of Definition \ref{def1} we have $\e_0 \le
\frac{1}{4}$.
\end{corollary}
\begin{proof}
Let $\delta>0$ be arbitrary and apply Theorem \ref{thm:main} with
$\e=\frac{1}{4}+\delta$ and $\vu_1=1$. Then $X_\e=(\e, 1-\e)$, and
$X_\e -X_\e$ is not dense in $\TT$. Therefore, $\e_0\le \e$, and
hence $\e_0\le \frac{1}{4}$.
\end{proof}

\medskip

A more elaborate argument gives the following improvement.

\begin{corollary}\label{onefifth}
With the notation of Definition \ref{def1} we have
$\e_0 \le \frac{1}{5}$.
\end{corollary}

\begin{proof}
Let $\delta>0$ arbitrary and apply Theorem \ref{thm:main} with
$\e=\frac{1}{5}+\delta$ and $\vu_1=1, \vu_2=e^{2\pi i/3},
\vu_3=e^{4\pi i/3}$. In this case $S\subset \TT^3$ is a
two-dimensional torus, which can be identified with $[0,1)^2$ via
the projection mapping $(s_1, s_2, s_3)\mapsto (s_1, s_2)$.
Elementary calculations give the exact position of the region
$X_\e$: it is the union of two triangles $T_1$ and $T_2$ with
coordinates of vertices $(\e, \e), (\e, 1-2\e), (1-2\e, \e)$ and
$(1-\e, 1-\e), (1-\e, 2\e), (2\e, 1-\e)$, respectively. However,
it is easy to check that these triangles are "too small" in the
sense that $(T_1\cup T_2)-(T_1\cup T_2)$ is not dense in
$[0,1)^2$. Therefore, $\e_0\le \e$, and hence $\e_0\le
\frac{1}{5}$.
\end{proof}

We conclude this note with two remarks.

\medskip

First, notice that the proofs of Corollaries \ref{onefourth} and
\ref{onefifth} follow the same pattern. If $\vu_1, \dots, \vu_n$
and $\e>0$ correspond to a covering of $\R^2$, we slightly
decrease the half-width of the strips to some $\rho<\e$, so that
the non-covered region is still small enough, and then apply
Theorem \ref{thm:main} to conclude that $\e_0\le \rho$. In
Corollary \ref{onefourth} this was done for the trivial covering
$\vu_1=1$, $\e=\frac{1}{2}$ with the choice $\rho=\frac{1}{4}$. In
Corollary \ref{onefifth} we used the covering $\vu_1=1,
\vu_2=e^{2\pi i/3}, \vu_3=e^{4\pi i/3}$ for $\e=\frac{1}{3}$, with
the choice $\rho=\frac{1}{5}$. It is plausible that this can be
done for {\it every} covering, thus reducing the value of $\e$
ever further. This would mean that the set of $\e$'s having the
finite rotation property is open. However, it would still not
prove that $\e_0=0$.

\medskip

Second, Theorem \ref{thm:main} gives us the possibility to
re-consider periodic coverings. For instance, let $\e<\frac{1}{2}$
and take all Pythagorean triples $m_j^2+k_j^2=n_j^2$ with $n_j\le
N$ for some fixed $N$. We know from Section \ref{sec:per} that the
corresponding angles will not give a covering for $\e$. However,
it is possible that the non-covered part $X$ of the plane is small
enough so that Theorem \ref{thm:main} can be invoked to conclude
that $\e$ has the finite rotation property. In fact, it is
possible that this argument works for any $\e>0$ if we choose $N$
large enough, but we could not prove it so far.

\end{document}